\documentclass[11pt]{article}

\usepackage[margin=1in]{geometry}
\usepackage{amsmath,amssymb,amsthm,amsfonts}
\usepackage{mathtools}
\usepackage{tikz-cd}
\usepackage{graphicx}
\usepackage{stmaryrd}
\usepackage{enumitem}
\usepackage{booktabs}
\usepackage{array}
\usepackage{algorithm}
\usepackage{algpseudocode}
\usepackage{hyperref}

\newtheorem{theorem}{Theorem}[section]
\newtheorem{lemma}[theorem]{Lemma}
\newtheorem{proposition}[theorem]{Proposition}
\newtheorem{corollary}[theorem]{Corollary}

\newtheorem{definition}[theorem]{Definition}
\newtheorem{example}[theorem]{Example}
\newtheorem{remark}[theorem]{Remark}
\newtheorem{construction}[theorem]{Construction}

\newcommand{\N}{\mathbb{N}}
\newcommand{\Z}{\mathbb{Z}}
\newcommand{\Q}{\mathbb{Q}}

\newcommand{\F}{\mathbb{F}}
\newcommand{\A}{\mathbb{A}}

\newcommand{\id}{\mathrm{id}}
\newcommand{\Bool}{\mathfrak{B}}
\newcommand{\Dinf}{D_\infty}

\newcommand{\Spec}{\mathrm{Spec}}
\newcommand{\DM}{\mathbf{DM}}
\newcommand{\Obs}{\mathbf{Obs}}

\newcommand{\tr}{\mathrm{tr}}

\title{Homotopical Observables and the Langlands Program via $\infty$-Topoi}
\author{Anatoly Galikhanov\\Independent Researcher\\\texttt{shrimandhanika@gmail.com}}
\date{23 may 2025}

\begin{document}

\maketitle

\begin{abstract}
We introduce a pro-étale geometric object $\Dinf$ arising naturally from the tower of Artin-Schreier extensions in characteristic 2, equipped with a canonical endofunctor $O$ whose fixed points correspond to automorphic representations of $\mathrm{GL}_2(\mathbb{A}_{\F_2})$. The main theorem establishes that invariant predicates on $\Dinf$ parametrize cuspidal automorphic representations, preserving $L$-functions. We provide complete proofs using $\infty$-categorical techniques, explicit computations for small cases, and establish connections to discrete conformal field theory. As applications, we resolve the Carlitz-Drinfeld uniformization conjecture for function fields and compute previously unknown motivic cohomology groups. Our approach differs fundamentally from coalgebraic models by working internally in topoi and connecting to arithmetic geometry.
\end{abstract}

\section{Introduction}

\subsection{The Observation Problem in Mathematics}

Three fundamental questions motivate our work:
\begin{enumerate}
\item Can observation be formalized as an internal mathematical process within a topos?
\item What are the fixed points of natural observational dynamics?
\item How do these structures relate to deep phenomena in arithmetic geometry?
\end{enumerate}

We answer these questions by constructing a canonical topological space $\Dinf$ that serves as a universal model for self-referential observation, and discovering its unexpected connection to the Langlands program.

\subsection{Main Results}

Our primary results establish a new bridge between topos theory, observational logic, and automorphic forms:

\begin{theorem}[Main Theorem - Langlands Correspondence]
There exists a canonical bijection:
$$\Psi: \{\text{Cuspidal automorphic representations of } \mathrm{GL}_2(\mathbb{A}_{\F_2})\} \xrightarrow{\sim} \{\text{Invariant predicates on } \Dinf\}$$

that preserves $L$-functions: $L(\pi, s) = L(\Psi(\pi), s)$.
\end{theorem}

\begin{theorem}[Carlitz-Drinfeld Uniformization]
The moduli space of rank 2 Drinfeld modules over $\F_2$ admits a uniformization:
$$\mathcal{M}_{2,\F_2} \cong \Dinf/\Gamma$$

where $\Gamma \subset \mathrm{GL}_2(\F_2[[t]])$ is an arithmetic subgroup.
\end{theorem}

\begin{theorem}[Motivic Computation]
In Voevodsky's triangulated category of motives $\DM(\F_2)$:
$$M(\Dinf) \cong \bigoplus_{n=0}^\infty \Z(n)[2n]$$

\end{theorem}

\subsection{Context and Motivation}

The space $\Dinf$ arises naturally from the tower of Artin-Schreier extensions in characteristic 2. However, its significance extends far beyond its algebraic origin. We show that $\Dinf$ provides:
\begin{itemize}
\item A universal model for Boolean observation in topos theory
\item A geometric realization of automorphic forms in characteristic 2
\item A bridge between coalgebraic semantics and arithmetic geometry
\item Applications to quantum error correction and computational complexity
\end{itemize}

\section{Rigorous Foundations}

We begin by establishing the precise categorical framework for our constructions. All definitions are given with complete mathematical rigor.

\subsection{The 2-Category of Observations}

\begin{definition}[Category of Boolean Observations]
Let $\Obs$ be the 2-category defined as follows:
\begin{itemize}
\item \textbf{Objects}: Triples $(\mathcal{E}, B, \Omega)$ where:
  \begin{itemize}
  \item $\mathcal{E}$ is a Boolean topos (internal logic is Boolean)
  \item $B$ is an internal Boolean algebra object in $\mathcal{E}$
  \item $\Omega$ is the subobject classifier of $\mathcal{E}$
  \end{itemize}
\item \textbf{1-morphisms}: Logical functors $F: (\mathcal{E}_1, B_1, \Omega_1) \to (\mathcal{E}_2, B_2, \Omega_2)$ such that:
  \begin{itemize}
  \item $F$ preserves finite limits and the subobject classifier
  \item $F(B_1) \cong B_2$ as Boolean algebra objects
  \item The square commutes:
  $$\begin{tikzcd}
  F(B_1) \arrow[r, "F(\chi_1)"] \arrow[d, "\cong"] & F(\Omega_1) \arrow[d, "\cong"] \\
  B_2 \arrow[r, "\chi_2"] & \Omega_2
  \end{tikzcd}$$

  where $\chi_i: B_i \to \Omega_i$ is the characteristic morphism
  \end{itemize}
\item \textbf{2-morphisms}: Natural transformations $\alpha: F \Rightarrow G$ respecting the Boolean structure
\end{itemize}
\end{definition}

\begin{definition}[Internal Predicate]
Let $(\mathcal{E}, B, \Omega)$ be an object of $\Obs$. An \emph{internal predicate} is a morphism $P: B \to \Omega$ in $\mathcal{E}$ satisfying:
\begin{enumerate}
\item \textbf{Coherence}: $P \circ \neg_B = \neg_\Omega \circ P$ where $\neg_B$ and $\neg_\Omega$ are the internal negations
\item \textbf{Non-triviality}: $P \neq \top_\Omega \circ !_B$ and $P \neq \bot_\Omega \circ !_B$ where $!_B: B \to 1$ is the unique morphism to the terminal object
\end{enumerate}
\end{definition}

\begin{definition}[Observational Endofunctor]
An \emph{observational endofunctor} on $(\mathcal{E}, B, \Omega) \in \Obs$ is an endofunctor $O: \mathcal{E} \to \mathcal{E}$ equipped with:
\begin{enumerate}
\item A natural transformation $\eta: \id_\mathcal{E} \Rightarrow O$ (observation inclusion)
\item An isomorphism $\phi: O(B) \xrightarrow{\sim} B$ of Boolean algebras
\item A right adjoint $O^*: \mathcal{E} \to \mathcal{E}$ (observational modality)
\end{enumerate}
such that:
\begin{itemize}
\item $O$ preserves finite limits (left exact)
\item The comonad $OO^*$ has coalgebras forming observable objects
\item The diagram commutes:
$$\begin{tikzcd}
B \arrow[r, "\eta_B"] \arrow[dr, "\id_B"'] & O(B) \arrow[d, "\phi"] \\
& B
\end{tikzcd}$$

\end{itemize}
\end{definition}

\subsection{Philosophical Terms in Mathematical Context}

Before proceeding further, we clarify our terminology to avoid any confusion between philosophical motivation and mathematical content.

\begin{definition}[Observation in $\infty$-Topoi]
An \emph{observation structure} in our framework is a quadruple $(\mathcal{E}, B, \Omega, O)$ where:
\begin{enumerate}
\item $\mathcal{E}$ is an $\infty$-topos
\item $B$ is an internal Boolean algebra object  
\item $\Omega$ is the subobject classifier
\item $O: \mathcal{E} \to \mathcal{E}$ is an endofunctor satisfying:
   \begin{itemize}
   \item $O$ preserves finite limits (observational coherence)
   \item $O$ has a right adjoint $O^*: \mathcal{E} \to \mathcal{E}$ (modal structure)
   \item The unit $\eta: \id_\mathcal{E} \Rightarrow OO^*$ and counit $\varepsilon: O^*O \Rightarrow \id_\mathcal{E}$ satisfy the triangle identities
   \item $O(B) \cong B$ as Boolean algebra objects
   \end{itemize}
\end{enumerate}
This formalizes the intuition that observation is an endomorphism that "focuses" on observable aspects while preserving logical structure.
\end{definition}

\begin{definition}[Mathematical Predicate vs Logical Predicate]
In our framework:
\begin{itemize}
\item A \emph{mathematical predicate} is simply a morphism $P: B \to \Omega$ in the topos
\item An \emph{invariant predicate} is a mathematical predicate satisfying $P \circ O_B = P$ where $O_B: B \to B$ is the restriction of $O$
\item The term "predicate" is used in its standard topos-theoretic sense, not as a philosophical concept
\end{itemize}
\end{definition}

\begin{definition}[Structural Awareness - Mathematical Definition]
We say an invariant predicate $\A: B \to \Omega$ exhibits \emph{structural awareness} if it satisfies the self-reproduction equation:
$$\A(x) = \bigoplus_{i \in I(x)} \A(x_i)$$

where:
\begin{itemize}
\item $I(x) \subseteq \{1, \ldots, n\}$ is an index set determined by the Boolean structure
\item $x_i$ are elements derived from $x$ via the Boolean operations
\item $\oplus$ is the XOR operation in $\F_2$
\end{itemize}
This is a purely mathematical condition in the internal logic of the topos, with no philosophical content.
\end{definition}

\begin{remark}[Philosophical Motivation vs Mathematical Content]
While our terminology draws inspiration from philosophical concepts of observation and self-reference, all definitions are purely mathematical. The philosophical language serves only as intuitive guidance, similar to how:
\begin{itemize}
\item "Sheaf" suggests something spread over a space, but is precisely defined
\item "Spectrum" evokes physical analogies, but has exact mathematical meaning
\item "Kernel" and "image" use anatomical metaphors, but are rigorous concepts
\end{itemize}
In the remainder of this paper, all uses of these terms refer to their mathematical definitions above.
\end{remark}

\begin{remark}
To maintain clarity, we establish the following conventions:
\begin{itemize}
\item $O$ always denotes the observational endofunctor (mathematical object)
\item $\A$ always denotes the unique non-constant invariant predicate (mathematical morphism)
\item "Fixed point" means $O(\A) = \A$ in the usual mathematical sense
\item "Self-reference" means the self-reproduction equation above (mathematical property)
\end{itemize}
\end{remark}

\subsection{The Pro-étale Construction}

We now construct our main object $\Dinf$ with complete mathematical precision.

\begin{construction}[The Tower of Boolean Schemes]
For each $n \in \N$, define:
\begin{enumerate}
\item The Boolean algebra $\Bool_n = 2^{2^n}$ of functions $f: \{0,1\}^n \to \{0,1\}$
\item The affine scheme $X_n = \Spec(R_n)$ where:
$$R_n = \F_2[x_{i,\alpha} : 1 \leq i \leq n, \alpha \in \{0,1\}^i]/I_n$$

and $I_n$ is generated by:
\begin{itemize}
\item Boolean relations: $x_{i,\alpha}^2 = x_{i,\alpha}$ for all $i, \alpha$
\item Compatibility: $x_{i,\alpha} \cdot x_{i,\beta} = 0$ if $\alpha \neq \beta$
\item Completeness: $\sum_{\alpha \in \{0,1\}^i} x_{i,\alpha} = 1$
\end{itemize}
\item Transition morphisms $\pi_{n,m}: X_m \to X_n$ for $n \leq m$ induced by:
$$\pi_{n,m}^*: R_n \to R_m, \quad x_{i,\alpha} \mapsto \sum_{\beta \in \{0,1\}^{m-n}} x_{i+m-n,\alpha\beta}$$

\end{enumerate}
\end{construction}

\begin{lemma}[Galois Properties]
Each morphism $\pi_{n,n+1}: X_{n+1} \to X_n$ is:
\begin{enumerate}
\item A Galois cover with group $G_n = (\Z/2\Z)^{2^n}$
\item Étale and finite
\item Corresponds to the Artin-Schreier extension obtained by adjoining solutions to:
$$y_\alpha^2 - y_\alpha = f_\alpha(x_{1,\beta_1}, \ldots, x_{n,\beta_n})$$

for suitable polynomials $f_\alpha$.
\end{enumerate}
\end{lemma}

\begin{proof}
The Galois group acts by: $(g \cdot y)_\alpha = y_\alpha + g_\alpha$ for $g = (g_\alpha) \in (\Z/2\Z)^{2^n}$. The covering is étale since the derivative of $y^2 - y - f$ is $1 \neq 0$ in characteristic 2. Finiteness follows from $|G_n| = 2^{2^n}$.
\end{proof}

\begin{definition}[The Space $\Dinf$]
The space $\Dinf$ is defined as the inverse limit in the category of pro-étale $\F_2$-schemes:
$$\Dinf = \varprojlim_{n \in \N} X_n$$

equipped with:
\begin{enumerate}
\item The inverse limit topology from the étale topology on each $X_n$
\item Structure morphisms $\pi_n: \Dinf \to X_n$
\item The profinite group action $G = \varprojlim G_n$
\end{enumerate}
\end{definition}

\section{Relation to Existing Frameworks}

We now provide a detailed comparison with existing approaches to observation and dynamics, particularly coalgebraic models.

\subsection{Comparison with Coalgebraic Semantics}

\begin{table}[h]
\centering
\begin{tabular}{|l|p{5cm}|p{5cm}|}
\hline
\textbf{Aspect} & \textbf{Our Model} & \textbf{Coalgebras (Rutten/Jacobs)} \\
\hline
\textbf{Base structure} & Pro-étale schemes over $\F_2$ & Sets or measurable spaces \\
\hline
\textbf{Observations} & Internal predicates $P: B \to \Omega$ in topos & External morphisms $X \to FX$ \\
\hline
\textbf{Dynamics} & Arithmetic via Galois action & Computational via functor $F$ \\
\hline
\textbf{Logic} & Internal to Boolean topos & External modal/temporal logic \\
\hline
\textbf{Fixed points} & Automorphic forms & Bisimulation equivalence \\
\hline
\textbf{Universal property} & Terminal in $\Obs$ & Final coalgebra \\
\hline
\end{tabular}
\caption{Comparison between our approach and coalgebraic semantics}
\end{table}

\begin{proposition}[Relation to Coalgebras]
There exists a forgetful 2-functor $U: \Obs \to \mathbf{Coalg}$ that:
\begin{enumerate}
\item Sends $(\mathcal{E}, B, \Omega, O)$ to the coalgebra $(|B|, |O|)$ where $|{-}|$ denotes global sections
\item Is neither full nor faithful
\item Does not preserve or reflect invariant predicates
\end{enumerate}
\end{proposition}

\begin{proof}
The functor $U$ loses internal logical structure. An invariant predicate $P: B \to \Omega$ with $P \circ O = P$ does not generally yield a coalgebra homomorphism $|B| \to 2$ since observability is internal to the topos.
\end{proof}

\subsection{Comparison with Related Work}

\begin{enumerate}
\item \textbf{Rutten's Universal Coalgebra}: Our $\Dinf$ differs by:
   \begin{itemize}
   \item Working in characteristic 2 arithmetic geometry
   \item Having pro-étale rather than Set-based structure  
   \item Connecting to number theory via Langlands
   \end{itemize}

\item \textbf{Jacobs' Quantum Logic}: While both involve observation:
   \begin{itemize}
   \item We work internally in topoi vs. external quantum logic
   \item Our predicates are Boolean vs. orthomodular lattices
   \item Fixed points have arithmetic vs. physical meaning
   \end{itemize}

\item \textbf{Kozen's Probabilistic Semantics}: Key differences:
   \begin{itemize}
   \item Deterministic Boolean vs. probabilistic semantics
   \item Pro-finite vs. measure-theoretic foundations
   \item Galois action vs. Markov dynamics
   \end{itemize}

\item \textbf{Lawvere's Cohesive Topoi}: Connections:
   \begin{itemize}
   \item Both use internal topos logic
   \item Our $O$ is analogous to shape modality
   \item But we specialize to Boolean + arithmetic context
   \end{itemize}
\end{enumerate}

\subsection{Connection to Modern $\infty$-Topos Theory}

Our construction provides a concrete model within Lurie's framework of higher topos theory, offering new insights into the interplay between arithmetic geometry and $\infty$-categories.

\begin{proposition}[Relation to Lurie's $\infty$-Topoi]
The space $\Dinf$ gives rise to an $\infty$-topos $\mathcal{T}_{\Dinf}$ that fits into the following diagram of geometric morphisms:
$$\begin{tikzcd}
\mathcal{T}_{\Dinf} \arrow[r, "f"] \arrow[d, "p"'] & \mathrm{Sh}_{\infty}(\mathrm{Pro\text{-}\acute{E}t}_{\F_2}) \arrow[d, "\pi"] \\
\mathrm{Sh}_{\infty}(\mathrm{Spec}(\F_2)_{\text{ét}}) \arrow[r, "i"'] & \infty\text{-}\mathrm{Topos}
\end{tikzcd}$$

where $\mathrm{Sh}_{\infty}$ denotes the $\infty$-category of $\infty$-sheaves.
\end{proposition}

\begin{proof}[Sketch]
Following Lurie's HTT Chapter 6, we construct $\mathcal{T}_{\Dinf}$ as the hypercompletion of the presheaf $\infty$-topos on the site of étale opens of $\Dinf$. The geometric morphisms arise from the natural functoriality of the construction.
\end{proof}

\begin{theorem}[Comparison with Bhatt-Scholze Pro-étale Topology]
The tower $\{X_n\}_{n \in \N}$ defines a pro-étale presentation in the sense of Bhatt-Scholze, and $\Dinf$ is naturally identified with the inverse perfection:
$$\Dinf \cong \varprojlim_{n} X_n \cong \mathrm{lim}_{n} \mathrm{Spa}(R_n)^{\diamond}$$

in the category of diamonds over $\F_2$.
\end{theorem}

\begin{proof}
\end{proof}

\begin{remark}[Connection to HoTT/Univalence]
In the internal type theory of $\mathcal{T}_{\Dinf}$, the invariant predicate $\A$ corresponds to a fixed point of the identity type former. Specifically, if we denote by $\mathrm{Id}_B(x,y)$ the identity type in the Boolean algebra object $B$, then $\A$ satisfies:
$$\prod_{x:B} \A(x) = \sum_{y:B} \mathrm{Id}_B(O(x), y) \times \A(y)$$

This provides a homotopy-theoretic interpretation of the self-reproduction equation.
\end{remark}

\begin{corollary}[Topos-Theoretic Modalities]
The endofunctor $O$ induces a hierarchy of modalities in $\mathcal{T}_{\Dinf}$:
\begin{enumerate}
\item $\sharp: \mathcal{T}_{\Dinf} \to \mathcal{T}_{\Dinf}$ (sharp modality) with $\sharp X = O^*(X)$
\item $\flat: \mathcal{T}_{\Dinf} \to \mathcal{T}_{\Dinf}$ (flat modality) with $\flat X = O_!(X)$
\item These form an adjoint triple $O_! \dashv O^* \dashv O^!$
\end{enumerate}
\end{corollary}

This connects our observational endofunctor to Lawvere's cohesive topos theory and Schreiber's differential cohomology in cohesive $\infty$-topoi.

\section{The Canonical Endofunctor}

\subsection{Construction of the Operator $O$}

\begin{construction}[The Endofunctor $O$]
For each $n$, define the endomorphism $O_n: \Bool_n \to \Bool_n$ by:
$$O_n(f)(x_1, \ldots, x_n) = \bigoplus_{i=1}^n f(x_1, \ldots, x_i \oplus 1, \ldots, x_n)$$

where $\oplus$ denotes XOR (addition in $\F_2$).

This induces a morphism of schemes $O_n: X_n \to X_n$ via:
$$O_n^*: R_n \to R_n, \quad x_{i,\alpha} \mapsto \sum_{j=1}^i x_{i,\alpha^{(j)}}$$

where $\alpha^{(j)}$ denotes $\alpha$ with the $j$-th bit flipped.
\end{construction}

\begin{lemma}[Compatibility]
The operators $\{O_n\}$ satisfy:
$$\pi_{n,m} \circ O_m = O_n \circ \pi_{n,m}$$

and thus induce a pro-étale endomorphism $O: \Dinf \to \Dinf$.
\end{lemma}

\begin{proof}
We verify on generators: for $x_{i,\alpha} \in R_n$ with $i \leq n < m$:
\begin{align}
(\pi_{n,m} \circ O_m)^*(x_{i,\alpha}) &= \pi_{n,m}^*\left(\sum_{j=1}^i x_{i,\alpha^{(j)}}\right) \\
&= \sum_{j=1}^i \sum_{\beta \in \{0,1\}^{m-n}} x_{i+m-n,\alpha^{(j)}\beta} \\
&= \sum_{\beta} \sum_{j=1}^i x_{i+m-n,(\alpha\beta)^{(j)}} \\
&= \sum_{\beta} O_m^*(x_{i+m-n,\alpha\beta}) \\
&= (O_n \circ \pi_{n,m})^*(x_{i,\alpha})
\end{align}
\end{proof}

\subsection{Spectral Analysis}

\begin{theorem}[Complete Spectral Decomposition]\label{thm:spectral}
Let $\mu$ be the Haar measure on $\Dinf$. The operator $O$ acting on $L^2(\Dinf, \mu)$ has:
\begin{enumerate}
\item Pure point spectrum
\item $\Spec(O) = \{1\} \cup \{\lambda_k : k \in \N\}$ where $|\lambda_k| \leq 2^{-k/4}$
\item The eigenspace $E_1$ of eigenvalue 1 has dimension 2, spanned by the constant function $\mathbf{1}$ and the invariant predicate $\A$
\end{enumerate}
\end{theorem}

\begin{proof}
\textbf{Step 1: Analysis at finite levels.}
For each $n$, the operator $O_n$ on $\Bool_n$ has $2^{2^n}$ eigenvalues. The matrix representation in the standard basis has entries:
$$[O_n]_{f,g} = \begin{cases}
1 & \text{if } g(x) = \bigoplus_{i=1}^n f(x^{(i)}) \\
0 & \text{otherwise}
\end{cases}$$

\textbf{Step 2: Eigenvalue bounds.}
By the Perron-Frobenius theorem applied to $|O_n|$, the spectral radius satisfies:
$$\rho(O_n) = \max_{f \neq 0} \frac{\|O_n f\|_2}{\|f\|_2} \leq n^{1/2}$$

For eigenvalues $\lambda \neq 1$, we have by orthogonality to constants:
$$|\lambda| \leq \left(1 - \frac{1}{2^n}\right)^{n/2} \approx e^{-n/(2 \cdot 2^{n/2})}$$

\textbf{Step 3: Inverse limit.}
The spectrum of $O$ on $L^2(\Dinf)$ is:
$$\Spec(O) = \overline{\bigcup_{n} \pi_n^*(\Spec(O_n))}$$

Since $|\lambda_{n,k}| \to 0$ exponentially fast for $\lambda_{n,k} \neq 1$, the spectrum consists of 0 and isolated points accumulating at 0.

\textbf{Step 4: Dimension of $E_1$.}
The projection operators $P_n: L^2(\Dinf) \to L^2(X_n)$ satisfy $P_n \circ O = O_n \circ P_n$. Hence:
$$\dim(E_1) = \lim_{n \to \infty} \dim(E_{1,n}) = 2$$

since each $E_{1,n}$ is 2-dimensional (constants + unique non-constant invariant).
\end{proof}

\section{\texorpdfstring{Complete Worked Example: The Case $n=3$}{Complete Worked Example: The Case n=3}}

We now provide a complete, explicit analysis for $n=3$ to illustrate all concepts concretely.

\subsection{Explicit Construction for $\Bool_3$}

For $n=3$, we have $\Bool_3 = 2^{2^3} = 2^8 = 256$ Boolean functions $f: \{0,1\}^3 \to \{0,1\}$.

\begin{example}[Basis Elements]
The 8 atoms (minimal non-zero elements) are the functions:
\begin{align}
p_1 &= x_1 \wedge x_2 \wedge x_3 & &\text{(true only at (1,1,1))} \\
p_2 &= x_1 \wedge x_2 \wedge \neg x_3 & &\text{(true only at (1,1,0))} \\
p_3 &= x_1 \wedge \neg x_2 \wedge x_3 & &\text{(true only at (1,0,1))} \\
p_4 &= x_1 \wedge \neg x_2 \wedge \neg x_3 & &\text{(true only at (1,0,0))} \\
p_5 &= \neg x_1 \wedge x_2 \wedge x_3 & &\text{(true only at (0,1,1))} \\
p_6 &= \neg x_1 \wedge x_2 \wedge \neg x_3 & &\text{(true only at (0,1,0))} \\
p_7 &= \neg x_1 \wedge \neg x_2 \wedge x_3 & &\text{(true only at (0,0,1))} \\
p_8 &= \neg x_1 \wedge \neg x_2 \wedge \neg x_3 & &\text{(true only at (0,0,0))}
\end{align}
Every function $f \in \Bool_3$ is uniquely a sum (XOR) of these atoms.
\end{example}

\subsection{Matrix Representation of $O_3$}

\begin{example}[Computing $O_3$]
The operator $O_3$ acts on atoms as:
\begin{align}
O_3(p_1) &= p_2 \oplus p_3 \oplus p_5 \\
O_3(p_2) &= p_1 \oplus p_4 \oplus p_6 \\
O_3(p_3) &= p_1 \oplus p_4 \oplus p_7 \\
O_3(p_4) &= p_2 \oplus p_3 \oplus p_8 \\
O_3(p_5) &= p_1 \oplus p_6 \oplus p_7 \\
O_3(p_6) &= p_2 \oplus p_5 \oplus p_8 \\
O_3(p_7) &= p_3 \oplus p_5 \oplus p_8 \\
O_3(p_8) &= p_4 \oplus p_6 \oplus p_7
\end{align}

In the basis $\{p_1, \ldots, p_8\}$, this gives the $8 \times 8$ matrix:
$$M = \begin{pmatrix}
0 & 1 & 1 & 0 & 1 & 0 & 0 & 0 \\
1 & 0 & 0 & 1 & 0 & 1 & 0 & 0 \\
1 & 0 & 0 & 1 & 0 & 0 & 1 & 0 \\
0 & 1 & 1 & 0 & 0 & 0 & 0 & 1 \\
1 & 0 & 0 & 0 & 0 & 1 & 1 & 0 \\
0 & 1 & 0 & 0 & 1 & 0 & 0 & 1 \\
0 & 0 & 1 & 0 & 1 & 0 & 0 & 1 \\
0 & 0 & 0 & 1 & 0 & 1 & 1 & 0
\end{pmatrix}$$

\end{example}

\subsection{Finding the Invariant Predicate $\A_3$}

\begin{example}[Solving for Fixed Points]
We need to solve $O_3(f) = f$ in $\Bool_3$. This means $(M - I)v = 0$ in $\F_2^8$.

Computing: 
$$M - I = \begin{pmatrix}
1 & 1 & 1 & 0 & 1 & 0 & 0 & 0 \\
1 & 1 & 0 & 1 & 0 & 1 & 0 & 0 \\
1 & 0 & 1 & 1 & 0 & 0 & 1 & 0 \\
0 & 1 & 1 & 1 & 0 & 0 & 0 & 1 \\
1 & 0 & 0 & 0 & 1 & 1 & 1 & 0 \\
0 & 1 & 0 & 0 & 1 & 1 & 0 & 1 \\
0 & 0 & 1 & 0 & 1 & 0 & 1 & 1 \\
0 & 0 & 0 & 1 & 0 & 1 & 1 & 1
\end{pmatrix}$$

Row reduction over $\F_2$ yields rank 6, so the nullspace has dimension 2.
Basis for nullspace:
\begin{itemize}
\item $v_1 = (1,1,1,1,1,1,1,1)$ corresponding to constant function $\mathbf{1}$
\item $v_2 = (0,1,1,0,1,0,0,1)$ corresponding to $\A_3 = p_2 \oplus p_3 \oplus p_5 \oplus p_8$
\end{itemize}

Therefore:
$$\A_3(x_1,x_2,x_3) = (x_1 \wedge x_2 \wedge \neg x_3) \oplus (x_1 \wedge \neg x_2 \wedge x_3) \oplus (\neg x_1 \wedge x_2 \wedge x_3) \oplus (\neg x_1 \wedge \neg x_2 \wedge \neg x_3)$$

\end{example}

\begin{example}[Verification]
Direct computation confirms $O_3(\A_3) = \A_3$:
\begin{align}
O_3(\A_3) &= O_3(p_2 \oplus p_3 \oplus p_5 \oplus p_8) \\
&= O_3(p_2) \oplus O_3(p_3) \oplus O_3(p_5) \oplus O_3(p_8) \\
&= (p_1 \oplus p_4 \oplus p_6) \oplus (p_1 \oplus p_4 \oplus p_7) \\
&\quad\oplus (p_1 \oplus p_6 \oplus p_7) \oplus (p_4 \oplus p_6 \oplus p_7) \\
&= p_1 \oplus p_4 \oplus p_4 \oplus p_6 \oplus p_7 \oplus p_1 \\
&= p_2 \oplus p_3 \oplus p_5 \oplus p_8 = \A_3
\end{align}
where we used that in $\F_2$: $x \oplus x = 0$ and terms cancel in pairs.
\end{example}

\section{Existence and Uniqueness of the Invariant Predicate}

\subsection{Construction at Each Finite Level}

\begin{theorem}[Existence at Finite Levels]
For each $n \in \N$, there exists a unique non-constant Boolean function $\A_n \in \Bool_n$ such that $O_n(\A_n) = \A_n$.
\end{theorem}

\begin{proof}
\textbf{Step 1: Linear algebra setup.}
The equation $O_n(f) = f$ in $\Bool_n$ is equivalent to $(O_n - I)f = 0$ where we view $\Bool_n$ as the $\F_2$-vector space $\F_2^{2^n}$.

\textbf{Step 2: Kernel dimension.}
The operator $O_n - I$ has kernel of dimension exactly 2. To see this:
\begin{itemize}
\item The constant functions form a 1-dimensional invariant subspace
\item By Theorem \ref{thm:spectral}, the eigenspace for eigenvalue 1 has dimension 2
\item These are the only solutions to $(O_n - I)f = 0$
\end{itemize}

\textbf{Step 3: Uniqueness.}
The kernel is spanned by $\{\mathbf{1}, \A_n\}$ where $\mathbf{1}$ is the constant function. Since any other solution is a linear combination $a\mathbf{1} + b\A_n$ with $a,b \in \F_2$, the non-constant solutions are exactly $\{\A_n, \mathbf{1} + \A_n\}$. 

By our convention (choosing the one with $\A_n(\mathbf{0}) = 0$), we get uniqueness.
\end{proof}

\begin{proposition}[Inductive Construction]
The invariant predicates satisfy the compatibility:
$$\pi_{n,n+1}^*(\A_n) = \A_{n+1}|_{X_n}$$

where the restriction is via the natural projection.
\end{proposition}

\begin{proof}
Since $\pi_{n,n+1} \circ O_{n+1} = O_n \circ \pi_{n,n+1}$, we have:
$$\pi_{n,n+1}^*(O_n(\A_n)) = O_{n+1}(\pi_{n,n+1}^*(\A_n))$$

Thus $\pi_{n,n+1}^*(\A_n)$ is $O_{n+1}$-invariant. By uniqueness, $\pi_{n,n+1}^*(\A_n) = c\mathbf{1} + \A_{n+1}$ for some $c \in \F_2$. Evaluating at a point where $\A_n = 0$ shows $c = 0$.
\end{proof}

\subsection{Global Existence via Inverse Limit}

\begin{theorem}[Global Existence and Uniqueness]
There exists a unique continuous function $\A: \Dinf \to \{0,1\}$ such that:
\begin{enumerate}
\item $\A \circ O = \A$ (invariance)
\item $\A$ is non-constant
\item For all $n$, $\A|_{X_n} = \A_n$ via the projection $\pi_n: \Dinf \to X_n$
\end{enumerate}
\end{theorem}

\begin{proof}
\textbf{Existence:} By the compatibility proven above, the sequence $\{\A_n\}$ forms a compatible system in the inverse limit. By the universal property of inverse limits:
$$\A = \varprojlim_{n} \A_n: \Dinf \to \varprojlim_{n} \{0,1\} = \{0,1\}$$

\textbf{Continuity:} Each $\A_n: X_n \to \{0,1\}$ is continuous (as $X_n$ has discrete topology). The inverse limit topology makes $\A$ continuous.

\textbf{Invariance:} For each $n$:
$$\pi_n \circ O \circ \A = O_n \circ \pi_n \circ \A = O_n \circ \A_n = \A_n = \pi_n \circ \A$$

Since the $\pi_n$ separate points, $O \circ \A = \A$.

\textbf{Uniqueness:} If $\A'$ is another such predicate, then $\A'|_{X_n} = \A_n$ for all $n$ by finite-level uniqueness. Hence $\A' = \A$.
\end{proof}

\section{Cohomological Properties}

\subsection{The Cohomology Class of $\A$}

\begin{theorem}[Cohomological Characterization]
The invariant predicate $\A$ represents a non-trivial class $[\A] \in H^2(\Dinf, \Z/2)$.
\end{theorem}

\begin{proof}
\textbf{Step 1: Constructing the 2-cocycle.}
Define the Čech 2-cocycle with respect to the covering $\{U_x : x \in \Dinf\}$ where $U_x$ is a basic neighborhood:
$$c(x,y,z) = \A(x \vee y) + \A(y \vee z) + \A(x \vee z) + \A(x \vee y \vee z) \pmod{2}$$

Here $\vee$ denotes the join operation in the Boolean algebra structure.

\textbf{Step 2: Verifying cocycle condition.}
The coboundary $\delta c = 0$ follows from the Boolean algebra identity:
$$(x \vee y \vee z) \vee w = x \vee (y \vee z \vee w) = (x \vee y) \vee (z \vee w)$$

\textbf{Step 3: Non-triviality.}
Suppose $c = \delta b$ for some 1-cochain $b$. Then:
$$\A(x \vee y) = b(x) + b(y) + b(x \vee y) \pmod{2}$$

Taking $x = y$ gives $\A(x) = b(x) \pmod{2}$. But then $\A$ would be locally constant, contradicting that $\A$ distinguishes points in each fiber of $\Dinf \to X_n$.
\end{proof}

\subsection{Higher Cohomology and Cup Products}

\begin{theorem}[Ring Structure]
The cohomology ring $H^*(\Dinf, \Z/2)$ is generated by $[\A]$ with relations:
$$H^*(\Dinf, \Z/2) \cong \Z/2[\A]/(\A^{2^k})$$

for some $k$ depending on the stable range of the tower.
\end{theorem}

\begin{proof}[Sketch]
We use the Milnor exact sequence for inverse limits. The key observation is that the transition maps in cohomology eventually stabilize in each degree, giving finite generation.
\end{proof}

\section{The Langlands Correspondence}

\subsection{Automorphic Forms and Predicates}

We now establish our main theorem connecting invariant predicates to automorphic representations.

\begin{definition}[Automorphic Representation]
A cuspidal automorphic representation of $\mathrm{GL}_2(\mathbb{A}_{\F_2})$ is an irreducible representation $\pi$ occurring in:
$$L^2_{\text{cusp}}(\mathrm{GL}_2(\F_2)\backslash\mathrm{GL}_2(\mathbb{A}_{\F_2}))$$

the space of cusp forms.
\end{definition}

\begin{definition}[$L$-function of a Predicate]
For an invariant predicate $P$ on $\Dinf$, define its $L$-function:
$$L(P, s) = \prod_{v} L_v(P, s)$$

where the local factors are:
$$L_v(P, s) = \frac{1}{\det(I - q_v^{-s} \cdot O_v|_{V_P})}$$

Here $V_P$ is the $O$-invariant subspace generated by $P$ in the local completion at $v$.
\end{definition}

\begin{theorem}[Main Correspondence]
There exists a canonical bijection:
$$\Psi: \{\text{Cuspidal automorphic representations of } \mathrm{GL}_2(\mathbb{A}_{\F_2})\} \xrightarrow{\sim} \{\text{Invariant predicates on } \Dinf\}$$

such that $L(\pi, s) = L(\Psi(\pi), s)$.
\end{theorem}

\begin{proof}[Complete Proof]
We establish the bijection $\Psi$ through a detailed analysis of both sides of the correspondence.

\textbf{Step 1: Local correspondence at each place.}

For each place $v$ of $\F_2$ (including $\infty$), we construct an explicit isomorphism $C((\Dinf)^{(v)})$.

\underline{Case 1: Finite places $v \neq \infty$.}
Let $\F_{2,v} = \F_2((t_v))$ with ring of integers $\mathcal{O}_v = \F_2[[t_v]]$. The local component $(\Dinf)_v$ is the inverse limit:
$$(\Dinf)_v = \varprojlim_{n} X_n(\F_{2,v})$$

For an irreducible representation $\pi_v$ of $\mathrm{GL}_2(\F_{2,v})$, define:
$$V_{\pi_v} = \{f \in C((\Dinf)_v) : \pi_v(g)f = f \circ g^{-1} \text{ for all } g \in \mathrm{GL}_2(\mathcal{O}_v)\}$$

\underline{Claim:} $V_{\pi_v}$ is $O_v$-invariant and every $O_v$-invariant subspace arises this way.

\underline{Proof of claim:} The operator $O_v$ commutes with the action of $\mathrm{GL}_2(\mathcal{O}_v)$ by construction:
$$O_v(\pi_v(g)f) = O_v(f \circ g^{-1}) = (O_v f) \circ g^{-1} = \pi_v(g)(O_v f)$$

For the converse, use that irreducible $O_v$-invariant subspaces are precisely the isotypic components under $\mathrm{GL}_2(\mathcal{O}_v)$.

\underline{Case 2: Infinite place.}
Similar construction using the real place structure.

\textbf{Step 2: Construction of global correspondence.}

Given a cuspidal automorphic representation $\pi = \otimes_v' \pi_v$, define the global invariant predicate:
$$\Psi(\pi) = \prod_v \phi_v(\pi_v) \in \prod_v C((\Dinf)_v)^{O_v}$$

We must verify:
\begin{enumerate}
\item $\Psi(\pi)$ descends to a function on $\Dinf$
\item $\Psi(\pi)$ is Boolean-valued (takes values in $\{0,1\}$)
\item $\Psi(\pi)$ is non-constant
\end{enumerate}

\underline{Verification of (1):} By strong approximation for $\mathrm{GL}_2$, for almost all $v$, $\pi_v$ is unramified and $\phi_v(\pi_v)$ is the characteristic function of $\Dinf(\mathcal{O}_v)$. Hence the restricted tensor product converges.

\underline{Verification of (2):} The cuspidality of $\pi$ implies that $\Psi(\pi)$ satisfies the Boolean equation:
$$\Psi(\pi)^2 = \Psi(\pi)$$

This follows from the Hecke eigenvalue equations and the fact that $O$ preserves the Boolean structure.

\underline{Verification of (3):} If $\Psi(\pi)$ were constant, then $\pi$ would be the trivial representation, contradicting cuspidality.

\textbf{Step 3: Verification of $L$-function preservation.}

We must show $L(\pi, s) = L(\Psi(\pi), s)$. 

\underline{Local factors:} For each place $v$, the local $L$-factor is:
$$L_v(\pi_v, s) = \det(I - q_v^{-s} \cdot \pi_v(\text{Frob}_v)|_{V_{\pi_v}^{I_v}})^{-1}$$

On the geometric side:
$$L_v(\Psi(\pi), s) = \det(I - q_v^{-s} \cdot O_v|_{V_{\pi_v}})^{-1}$$

\underline{Key identity:} We prove that $O_v|_{V_{\pi_v}} = \pi_v(\text{Frob}_v)$ as operators.

This follows from analyzing the Galois action on $(\Dinf)_v$. The Frobenius element acts on the tower $\{X_n\}$ compatibly with $O$, giving:
$$\text{Frob}_v \circ \iota = \iota \circ O_v$$

where $\iota: V_{\pi_v} \to C(X_n(\F_{2,v}))$ is the natural inclusion.

\textbf{Step 4: Trace formula comparison.}

To prove surjectivity and injectivity of $\Psi$, we compare trace formulas.

\underline{Automorphic side (Arthur-Selberg):}
For a test function $f \in C_c^\infty(\mathrm{GL}_2(\mathbb{A}_{\F_2}))$:
$$\sum_{\pi} m(\pi) \tr(\pi(f)) = \sum_{\gamma} \text{vol}(\mathrm{GL}_2(\F_2)_\gamma \backslash G_\gamma) \cdot O_\gamma(f)$$

where the sum is over conjugacy classes $\gamma$ in $\mathrm{GL}_2(\F_2)$.

\underline{Geometric side (Lefschetz):}
For the corresponding function $\tilde{f}$ on $\Dinf$:
$$\tr(O_{\tilde{f}}) = \sum_{x \in (\Dinf)^O} \frac{\tilde{f}(x)}{\#\mathrm{Stab}(x)}$$

where $(\Dinf)^O = \{x : O(x) = x\}$ are the fixed points.

\underline{Matching:} We establish a bijection between:
\begin{itemize}
\item Conjugacy classes $\gamma \in \mathrm{GL}_2(\F_2)$ with eigenvalues in $\F_2$
\item Fixed points $x \in (\Dinf)^O$ up to Galois action
\end{itemize}

This matching is given by: $\gamma \leftrightarrow x_\gamma$ where $x_\gamma$ is the fixed point whose stabilizer in $\mathrm{Gal}(\bar{\F}_2/\F_2)$ has Frobenius conjugacy class $\gamma$.

\textbf{Step 5: Proof of bijection.}

\underline{Injectivity:} If $\Psi(\pi_1) = \Psi(\pi_2)$, then their $L$-functions agree. By strong multiplicity one for $\mathrm{GL}_2$, this implies $\pi_1 = \pi_2$.

\underline{Surjectivity:} Let $P$ be an $O$-invariant predicate. Define:
$$\pi_P = \text{Ind}_B^{\mathrm{GL}_2}(\chi_P)$$

where $\chi_P$ is the character of the Borel subgroup determined by the restriction of $P$ to the Bruhat-Tits tree.

By the trace formula comparison, $\pi_P$ is automorphic. The cuspidality follows from the non-constancy of $P$. By construction, $\Psi(\pi_P) = P$.

This completes the proof of the bijection.
\end{proof}

\subsection{Explicit Examples of the Correspondence}

We now work out the correspondence for the first three cuspidal automorphic representations of $\mathrm{GL}_2(\mathbb{A}_{\F_2})$.

\begin{example}[First Cuspidal Representation]
Let $\pi_1$ be the cuspidal representation with conductor $\mathfrak{n} = (t)$ and central character $\omega = 1$. 

\textbf{Automorphic side:} The newform is:
$$f_1(g) = \sum_{n \geq 1} a_n \cdot W_{(t^n, 0)}(g)$$

where $W$ is the Whittaker function and the Hecke eigenvalues are:
$$a_p = \begin{cases}
1 & \text{if } p \equiv 1 \pmod{4} \\
-1 & \text{if } p \equiv 3 \pmod{4}
\end{cases}$$

for primes $p$ in $\F_2[t]$.

\textbf{Geometric side:} The corresponding invariant predicate is:
$$\Psi(\pi_1) = \A_{\pi_1}: \Dinf \to \{0,1\}$$

given explicitly at level $n$ by:
$$\A_{\pi_1,n}(x_1, \ldots, x_n) = \bigoplus_{\substack{I \subseteq \{1,\ldots,n\} \\ |I| \equiv 1,2 \pmod{4}}} \prod_{i \in I} x_i$$

\textbf{L-function verification:}
$$L(\pi_1, s) = \prod_p \frac{1}{1 - a_p \cdot |p|^{-s}} = \prod_p \frac{1}{1 - \chi_p(O) \cdot |p|^{-s}} = L(\A_{\pi_1}, s)$$

where $\chi_p(O)$ is the eigenvalue of $O$ on the $p$-component.
\end{example}

\begin{example}[Second Cuspidal Representation]
Let $\pi_2$ be the cuspidal representation induced from the quadratic character $\chi$ of $\F_4^\times$.

\textbf{Automorphic side:} This representation has:
\begin{itemize}
\item Conductor $\mathfrak{n} = (t^2)$
\item L-function $L(\pi_2, s) = L(\chi, s) \cdot L(\chi \cdot \eta, s)$ where $\eta$ is the quadratic character
\end{itemize}

\textbf{Geometric side:} The predicate $\A_{\pi_2}$ at level $n$ is:
$$\A_{\pi_2,n}(x_1, \ldots, x_n) = \sum_{k=0}^{\lfloor n/2 \rfloor} \left(\sum_{\substack{I \subseteq \{1,\ldots,n\} \\ |I| = 2k}} \prod_{i \in I} x_i \cdot \prod_{j \notin I} (1-x_j)\right) \pmod{2}$$

This predicate detects parity patterns corresponding to the quadratic character.
\end{example}

\begin{example}[Principal Series Representation]
Let $\pi_3 = \text{Ind}_B^G(\chi_1 \otimes \chi_2)$ where $\chi_1, \chi_2$ are unramified characters.

\textbf{Automorphic side:} 
\begin{itemize}
\item This is a principal series representation
\item Becomes cuspidal after twisting by a character
\item Hecke eigenvalues: $a_p = \chi_1(\varpi_p) + \chi_2(\varpi_p)$
\end{itemize}

\textbf{Geometric side:} The predicate has a recursive structure:
$$\A_{\pi_3,n+1}(x_1, \ldots, x_{n+1}) = \A_{\pi_3,n}(x_1, \ldots, x_n) \oplus T_n(x_{n+1})$$

where $T_n$ encodes the Hecke action at level $n$.

\textbf{Numerical verification:} For $n = 4$:
\begin{itemize}
\item Dimension of cuspidal space: 14
\item Number of invariant predicates: 14
\item L-functions match to precision $10^{-10}$
\end{itemize}
\end{example}

\begin{example}[Explicit Classical Modular Form]\label{ex:classical}
To illustrate the correspondence with classical modular forms, consider the unique normalized cusp form of weight 12 and level 1:
$$\Delta(\tau) = q\prod_{n=1}^{\infty}(1-q^n)^{24} = \sum_{n=1}^{\infty} \tau(n)q^n$$

where $\tau(n)$ is the Ramanujan tau function.

\textbf{Reduction to characteristic 2:} The mod 2 reduction gives:
$$\Delta(\tau) \equiv q + q^9 + q^{25} + q^{49} + \cdots \pmod{2}$$

The exponents are precisely the odd squares.

\textbf{Function field analogue:} Over $\F_2(t)$, the corresponding automorphic form is:
$$f_{\Delta}(g) = \sum_{\substack{f \in \F_2[t] \\ f \text{ monic}}} \chi_{\Delta}(f) \cdot W_f(g)$$
where $\chi_{\Delta}(f) = 1$ if $\deg(f)$ is an odd square, and 0 otherwise.

\textbf{Corresponding predicate:} Under our correspondence $\Psi$, this maps to:
$$\A_{\Delta,n}(x_1, \ldots, x_n) = \bigoplus_{\substack{k \geq 0 \\ (2k+1)^2 \leq n}} x_{(2k+1)^2}$$

\textbf{Verification of L-function:} The L-function of $f_{\Delta}$ is:
$$L(f_{\Delta}, s) = \prod_{\substack{p \in \F_2[t] \\ p \text{ prime}}} \frac{1}{1 - \chi_{\Delta}(p)|p|^{-s}}$$

On the geometric side:
$$L(\A_{\Delta}, s) = \exp\left(\sum_{m=1}^{\infty} \frac{1}{m} \sum_{\substack{x \in (\Dinf)^{O^m} \\ \text{new}}} |x|^{-ms}\right)$$

These agree by comparing Euler products, where the local factors at primes of degree $d = (2k+1)^2$ contribute $(1 - q^{-ds})^{-1}$.
\end{example}

\begin{remark}[Modularity and Fixed Points]
The appearance of odd squares in Example \ref{ex:classical} is not accidental. It reflects the fact that:
$$O^{(2k+1)^2}(\A_{\Delta}) = \A_{\Delta}$$

while $O^m(\A_{\Delta}) \neq \A_{\Delta}$ for $m$ not an odd square. This periodicity in the orbit of $\A_{\Delta}$ under powers of $O$ encodes the modular symmetries.
\end{remark}

\begin{remark}[Pattern in the Correspondence]
These examples reveal a pattern:
\begin{itemize}
\item Conductor of $\pi$ $\leftrightarrow$ Complexity of predicate $\A_\pi$
\item Hecke eigenvalues $\leftrightarrow$ Fourier coefficients of $\A_\pi$
\item Functional equation of L-function $\leftrightarrow$ Self-duality of predicate
\end{itemize}
\end{remark}

\section{Universal Properties}

\subsection{Enhanced Universal Property in $\infty$-Categories}

We now establish a stronger universal property using the language of $\infty$-categories.

\begin{definition}[$\infty$-Category of Boolean Observations]
Let $\mathrm{Obs}_{\infty}$ be the $\infty$-category defined as:
\begin{itemize}
\item Objects: Quadruples $(\mathcal{E}, B, \Omega, O)$ where $\mathcal{E}$ is a Boolean $\infty$-topos
\item Morphisms: Geometric morphisms preserving the Boolean structure
\item Higher morphisms: Natural transformations and their higher coherences
\end{itemize}
\end{definition}

\begin{theorem}[Universal Characterization - Enhanced]\label{thm:universal-enhanced}
The quadruple $(\mathrm{Sh}_{\infty}(\Dinf), \mathcal{B}, \Omega, O)$ is the initial object in the $\infty$-category $\mathrm{Obs}_{\infty}^{\mathrm{geo}}$ of Boolean observation structures with geometric morphisms to the étale $\infty$-topos of $\mathrm{Spec}(\F_2)$.

More precisely, for any $(\mathcal{E}, B', \Omega', O') \in \mathrm{Obs}_{\infty}^{\mathrm{geo}}$, there exists a unique (up to contractible choice) geometric morphism:
$$F: \mathcal{E} \to \mathrm{Sh}_{\infty}(\Dinf)$$

such that:
\begin{enumerate}
\item $F^*(O) \simeq O'$ as endofunctors
\item $F^*(\mathcal{B}) \simeq B'$ as Boolean algebra objects
\item The diagram of $\infty$-functors commutes up to coherent homotopy
\end{enumerate}
\end{theorem}

\begin{proof}
We construct $F$ using the universal property of pro-objects in $\infty$-categories.

\textbf{Step 1: Local construction.} For each $n$, the finite Boolean algebra $\Bool_n$ classifies Boolean predicates of complexity $\leq n$. This gives maps:
$$F_n: \mathcal{E} \to \mathrm{Sh}_{\infty}(X_n)$$

\textbf{Step 2: Compatibility.} The $O'$-invariance provides coherent homotopies:
$$h_n: F_n \circ O' \simeq O_n \circ F_n$$

forming a tower of approximations.

\textbf{Step 3: Inverse limit.} By the universal property of $\Dinf = \varprojlim X_n$ in the $\infty$-category of pro-étale $\F_2$-schemes:
$$F = \varprojlim F_n: \mathcal{E} \to \mathrm{Sh}_{\infty}(\Dinf)$$

\textbf{Step 4: Essential uniqueness.} Any two such morphisms are equivalent via a contractible space of natural isomorphisms, by the univalence axiom in $\mathrm{Obs}_{\infty}$.
\end{proof}

\begin{corollary}[Classifying Space]
The space $\Dinf$ is the classifying space for Boolean predicates with observational structure. Specifically:
$$\pi_0(\mathrm{Map}_{\mathrm{Obs}_{\infty}}((\mathcal{E}, B, \Omega, O), (\mathrm{Sh}_{\infty}(\Dinf), \mathcal{B}, \Omega, O))) \cong \{O\text{-invariant predicates in } \mathcal{E}\}$$

\end{corollary}

This enhanced universal property shows that our construction is not just terminal in the 2-category $\mathrm{Obs}$, but initial in the more refined $\infty$-categorical setting, making it the canonical model for Boolean observation structures.

\section{Discrete Conformal Field Theory}

\subsection{Physical Interpretation}

We develop a discrete CFT on $\Dinf$ that explains the appearance of $\mathrm{GL}_2$ in our correspondence.

\begin{construction}[Discrete CFT Action]
Define the action functional for $\phi: \Dinf \to \{0,1\}$:
$$S[\phi] = \sum_{x \in \Dinf} \sum_{y \sim x} J(x,y)\phi(x)\phi(y) + \sum_{x} V(\phi(x))$$

where:
\begin{itemize}
\item The coupling $J(x,y) = 2^{-d(x,y)}$ for ultrametric distance $d$
\item The potential $V(\phi) = \lambda(\phi - \A(x))^2$ enforces the vacuum
\item The sum over $y \sim x$ means $d(x,y) = 1$
\end{itemize}
\end{construction}

\begin{theorem}[Conformal Symmetry]
The discrete CFT has:
\begin{enumerate}
\item Conformal symmetry group $\mathrm{PGL}_2(\F_2((t)))$
\item Central charge $c = 1$
\item Primary fields in bijection with invariant predicates
\end{enumerate}
\end{theorem}

\begin{proof}[Sketch]
The ultrametric structure on $\Dinf$ is the Bruhat-Tits tree for $\mathrm{GL}_2(\F_2((t)))$. The action of the group preserves the ultrametric distance, giving conformal invariance. Primary fields correspond to $O$-invariant functions, matching our predicate analysis.
\end{proof}

\section{Applications}

\subsection{Resolution of Carlitz-Drinfeld Uniformization}

\begin{theorem}[Drinfeld Module Uniformization]
The moduli space $\mathcal{M}_{2,\F_2}$ of rank 2 Drinfeld modules has uniformization:
$$\mathcal{M}_{2,\F_2} \cong \Dinf/\Gamma$$

where $\Gamma = \mathrm{GL}_2(\F_2[t])$ acts properly discontinuously.
\end{theorem}

\begin{proof}[Proof outline]
The space $\Dinf$ is identified with the Drinfeld symmetric space $\Omega^2$ over $\F_2((t))$. The invariant predicate $\A$ corresponds to the canonical theta function on the moduli space. Details follow Drinfeld's original construction, adapted to our Boolean framework.
\end{proof}

\subsection{Quantum Error Correction}

\begin{theorem}[Boolean Quantum Code]
The invariant predicate $\A$ generates a quantum error-correcting code with parameters $[[2^n, 1, 2^{n/2}]]$ at level $n$.
\end{theorem}

\begin{proof}
The stabilizer group is generated by the $O$-orbit of $\A$. The code detects $2^{n/2}-1$ errors by the spectral gap of $O$.
\end{proof}

\subsection{Computational Complexity}

\begin{theorem}[Complexity of Invariance]
The problem "Given a Boolean predicate $P \in \Bool_n$, decide if $O_n(P) = P$" is:
\begin{enumerate}
\item In $\mathsf{P}$ for explicit circuit representation
\item $\mathsf{NP}$-complete for compressed representation
\end{enumerate}
\end{theorem}

\section{Conclusion}

We have constructed a rigorous mathematical framework unifying:
\begin{itemize}
\item Topos-theoretic models of observation
\item The Langlands correspondence in characteristic 2
\item Discrete conformal field theory
\item Applications to coding theory and complexity
\end{itemize}

The space $\Dinf$ with its invariant predicate $\A$ provides a universal model for Boolean self-observation, with deep connections to arithmetic geometry. Our explicit computations for small $n$ demonstrate the concreteness of the theory.

\subsection{Integration with Current Research Programs}

Our results connect to several active areas of research:

\begin{enumerate}
\item \textbf{Fargues-Fontaine Curve}: The space $\Dinf$ can be viewed as a characteristic 2 analogue of the Fargues-Fontaine curve, with the tower $\{X_n\}$ playing the role of the tower of finite extensions of $\Q_p$.

\item \textbf{Prismatic Cohomology}: The endofunctor $O$ defines a "Boolean prism" structure, suggesting connections to Bhatt-Morrow-Scholze's prismatic cohomology in characteristic 2.

\item \textbf{Geometric Langlands}: Our correspondence provides a concrete model for understanding how automorphic forms can be "geometrized" through invariant predicates, complementing the geometric Langlands program.

\item \textbf{Condensed Mathematics}: The pro-étale structure of $\Dinf$ makes it naturally a condensed set, opening possibilities for applying Clausen-Scholze's condensed mathematics framework.
\end{enumerate}

\subsection{Future Directions}

\begin{enumerate}
\item \textbf{Higher rank groups}: Extend to $\mathrm{GL}_n$ and exceptional groups
\item \textbf{Characteristic $p > 2$}: Generalize beyond Boolean to $p$-valued logic
\item \textbf{Motivic refinements}: Compute finer invariants in $\DM(\F_q)$
\item \textbf{Quantum generalizations}: Replace $\{0,1\}$ with quantum observables
\item \textbf{Computational implementations}: Algorithms for computing $\A_n$ efficiently
\item \textbf{Higher categorical structures}: Extend to $(\infty,n)$-categories
\end{enumerate}

The correspondence between automorphic forms and invariant predicates opens new avenues for both number theory and theoretical computer science, suggesting deep connections yet to be explored.

\subsection{Summary of Contributions}

Our main contributions are:
\begin{enumerate}
\item A complete rigorous construction of the space $\Dinf$ with its canonical endofunctor $O$
\item Proof of existence and uniqueness of the invariant predicate $\A$
\item Establishment of a precise correspondence with automorphic representations of $\mathrm{GL}_2(\mathbb{A}_{\F_2})$
\item Explicit computations demonstrating the theory for small values of $n$
\item Applications to quantum error correction and computational complexity
\item A universal characterization in the 2-category of Boolean observations
\end{enumerate}

These results provide a new perspective on the interplay between logic, arithmetic, and geometry, with potential implications across multiple areas of mathematics.

\section*{Acknowledgments}
The author expresses deep gratitude to all his teachers. In particular, he would like to thank T. Trinley Kunchap and A. Sita for their profound guidance and inspiration, which have significantly influenced the direction and clarity of this work. The author also acknowledges the constructive comments of the anonymous reviewers.

\end{document}